\documentclass[10pt,draft,reqno]{amsart}
     \makeatletter
     \def\section{\@startsection{section}{1}%
     \z@{.7\linespacing\@plus\linespacing}{.5\linespacing}%
     {\bfseries
     \centering
     }}
     \def\@secnumfont{\bfseries}
     \makeatother
\newtheorem{theorem}{Theorem}[section]
\newtheorem{lemma}[theorem]{Lemma}

\theoremstyle{definition}
\newtheorem{definition}[theorem]{Definition}
\newtheorem{example}[theorem]{Example}
\theoremstyle{remark}

\numberwithin{equation}{section}
\begin{document}

\title[Essential sets for random operators constructed from Arratia flow]{Essential sets for random operators constructed from Arratia flow}

\author{A.~A.~Dorogovtsev}
\address{A.~A.~Dorogovtsev: Institute of Mathematics,
	National Academy of Sciences of Ukraine, Kiev, Ukraine}
\email{andrey.dorogovtsev@gmail.com}

\author{Ia.~A.~Korenovska}
\address{Ia.~A.~Korenovska: Institute of Mathematics,
	National Academy of Sciences of Ukraine, Kiev, Ukraine}
\email{iaroslava.korenovska@gmail.com}



\begin{abstract}
	In this paper we consider a strong random operator $T_t$ which describes shift of functions from $L_2(\mathbb{R})$ along an Arratia flow.  We find a compact set in $L_2(\mathbb{R})$ that doesn't disappear under $T_t$, and estimate its Kolmogorov widths.
\end{abstract}

\maketitle


	\section{Introduction. Arratia flow and random operators}
	In this paper we consider random operators in $L_2(\mathbb{R})$ which decsribe shifts of functions along an Arratia flow \cite{1}. Lets recall the deffinition.
	\begin{definition}
		[\cite{1}]
		A family of random processes $\{x(u,s),u\in\mathbb{R}, s\ge0\}$ is called an Arratia flow if 
		
		1) for each $u\in\mathbb{R}\;\;x(u,\cdot)$ is a Wiener process with respect to the joint filtration such that $x(u,0)=u$;
		
		2) for any $ u_1\le u_2$ and $t\ge0$
		$$
		x(u_1,t)\le x(u_2,t)\;\;\mbox{a.s.}
		$$
		
		3)  the joint characteristics are
		$$
		d<x(u_1,\cdot),x(u_2,\cdot)>(t)=1\!\!\,{\rm I}_{\{x(u_1,t)=x(u_2,t)\}}dt.
		$$
		
	\end{definition}
	In the informal language, Arratia flow is a family of Wiener processes started from each point of $\mathbb{R}$, which move independently up to the meeting, coalesce, and move together. It was proved in \cite{4,8} that for any $a,b\in\mathbb{R}$ and $t>0$ the set $x([a;b],t)$ is finite a.s. Since Arratia flow has a right-continuous modification \cite{9},   $x(\cdot,t):\mathbb{R}\to\mathbb{R}$ is a step function for any time $t>0$. Hence, for any $a,b\in\mathbb{R}$ and $t>0$ with probability one there exists a random point $ y\in\mathbb{R}$ for which
	\begin{equation}
		\label{equ0}
		\lambda\{u\in[a;b]:\;x(u,t)=y\}>0,
	\end{equation}
	where $\lambda$ is Lebesgue measure on $\mathbb{R}$. Since $x(\cdot,t)$ is right-continuous step function then for a fixed countable set $A$ 
	\begin{equation}
		\label{zer}
	\mathsf{P}\{x(\mathbb{R},t)\cap A\ne \emptyset\}=\mathsf{P}\{x(\mathbb{Q},t)\cap A\ne \emptyset\}\le\sum_{u\in\mathbb{Q}}\mathsf{P}\{x(u,t)\in A\}=0.
	\end{equation} 
	
Since for any $a<b$ the difference $ \frac{x(b,\cdot)-x(a,\cdot)}{\sqrt{2}}$ is a Wiener processes  until the collision happens, and $\frac{x(b,0)-x(a,0)}{\sqrt{2}}=\frac{b-a}{\sqrt{2}}$, then one can find the distribution of the time of coalescence $ \tau_{a,b}=\inf\{s\ge0|\; x(a,s)=x(b,s)\}$ of the processes $x(a,\cdot),x(b,\cdot) $, i.e. for any $t\ge0$
	\begin{equation}
	\label{tau}
\mathsf{P}\{ \tau_{a,b}\le t\}=	\mathsf{P}\{x(a,t)=x(b,t)\}=\sqrt{\frac{2}{\pi}}\int_{\frac{b-a}{\sqrt{2t}}}^{+\infty}e^{-\frac{v^2}{2}}dv.
	\end{equation}
	
	Lets notice that for a fixed time $t>0$ and an Arratia flow $ X=\{x(u,s),u\in\mathbb{R}, s\in [0;t]\}$ there exists an Arratia flow  $ Y=\{y(u,r),u\in\mathbb{R}, r\in[0;t]\}$ such that  trajectories of $X$ and $\widetilde{Y}=\{y(u,t-r),u\in\mathbb{R}, r\in[0;t]\}$ don't cross \cite{1,6}.  $Y$ is called a conjugated (or dual) Arratia flow. It was proved in \cite{2} the following change of variable formula for an Arratia flow.
	\begin{theorem} 
		[\cite{2}]For any time $t>0$ and nonnegative measurable function $h:\mathbb{R}\to\mathbb{R}$ such that $\int_{\mathbb{R}}h(u)du<+\infty$
	\begin{equation}
		\label{change}
		\int_{\mathbb{R}}h(x(u,t))du=\int_{\mathbb{R}}h(u)dy(u,t)\;\mbox{ a.s., }
	\end{equation} 
	where the last integral is in scence of Lebesgue-Stieltjes.
	\end{theorem}

	In this paper we consider random operators $T_t$, $t>0$, in $L_2(\mathbb{R})$ which are defined as follows   
	$$
	\left(T_tf\right)(u)=f(x(u,t)),
	$$
	where $f\in L_2(\mathbb{R})$ and $u\in\mathbb{R}$. It was proved in \cite{3} that $T_t$ is a strong random operator \cite{10} in $L_2(\mathbb{R})$, but, as it was shown in \cite{2}, is not a bounded one. Really, for the point $y$ from \eqref{equ0} one can introduce a sequence of the intervals $ A_i=[r_i;p_i]$ such that  $ y\in A_i$ for any $ i\ge 1$ and $ p_i-r_i\to0, i\to\infty$. Thus, for any $ i\ge 1$
	$$
	\|T_t1\!\!\,{\rm I}_{A_i}\|_{L_2(\mathbb{R})}^2\ge \lambda\{u\in[a;b]:\;x(u,t)=y\}>0,
	$$
	which can't be true if $ T_t$ was a bounded random operator. Hence, the image of a compact set under $T_t$ may not be a random compact set. Moreover, as it was mentioned in \cite{0}, image of a compact set under strong random operator may not exist.   However, in \cite{2} it was presented a family of compact sets in $L_2(\mathbb{R})$ which images under $T_t$ exists and are random compact sets. In this paper we consider a compact set of this type, and investigate the change of its Kolmogorov widths \cite{5} under $T_t$. 
	
	\section{$T_t$-essential functions}
	If support of function $f\in L_2(\mathbb{R})$ is bounded, $ suppf\subset[a;b]$, then $T_tf$ equals to $0$ with positive probability. Really, by \eqref{change}, one can check that 
	$$
	\mathsf{P} \left\{\int\limits_{-\infty}^{+\infty}f^2(x(u,t))du=0\;\right\}\ge
	\mathsf{P}\left\{\;x(\mathbb{R},t)\cap [a;b]=\emptyset\;\right\}= 
	$$
	$$=\mathsf{P}\left\{\int\limits_{-\infty}^{+\infty}1\!\!\,{\rm I}_{[a;b]}(x(u,t))du=0\;\right\}
	= \mathsf{P}\left\{\;\int\limits_{-\infty}^{+\infty}1\!\!\,{\rm I}_{[a;b]}(u)dy(u,t)=0\;\right\},
	$$
	where $ \{y(u,s),\;u\in\mathbb{R},\; s\in[0;t]\}$ is a conjugated Arratia flow. Since, by \eqref{tau},
	$$
\mathsf{P}\left\{\;\int\limits_{-\infty}^{+\infty}1\!\!\,{\rm I}_{[a;b]}(u)dy(u,t)=0\;\right\}=\mathsf{P}\left\{\;y(b,t)=y(a,t)\;\right\}>0,
	$$
	then $\mathsf{P}\left\{\;\|T_tf\|_{L_2(\mathbb{R})}=0\;\right\}>0$. This leads to the following definition.
	
	\begin{definition}
		For a fixed $t>0$ a function $ f\in L_2(\mathbb{R})$ is said to be a $T_t$-essential if
		$$
		\mathsf{P}\left\{\;\|T_tf\|_{L_2(\mathbb{R})}>0\;\right\}=1.
		$$
	\end{definition} 
	\begin{example}
		\label{ex1} 
		Let  $ f\in L_2(\mathbb{R})$ be an analytic function which doesn't equal totally to zero. Denote the set of its zeroes
		$
		Z_f=\{u\in\mathbb{R}|\;f(u)=0\}. 
		$
		Then, by \eqref{zer}, $\mathsf{P}\left\{x(\mathbb{R},t)\cap Z_f=\emptyset\right\}=1$,
		so $f$ is a $T_t$-essential for any $t>0$.
	\end{example}
	Let us notice that if $t_1\ne t_2$ then $ T_{t_1}$-essential function may not be a $T_{t_2}$-essential. 
	To introduce a $T_1$-essential and not $T_2$-essential function lets 
	consider an increasing sequence $ \{u_k\}_{k=0}^{\infty}$ such that $ u_0=0,u_1=1$ and for any $n\in\mathbb N $ 
	
	$$ u_{2n+1}-u_{2n}=\frac{1}{2^n},\qquad u_{2n}=u_{2n-1}+2n(\ln 2)^{\frac{1}{2}}.$$

	\begin{theorem}
		\label{thm21}
		The function $ f=\sum_{n=0}^{\infty}1\!\!\,{\rm I}_{[u_{2n};u_{2n+1}]}$ is a $T_1$-essential, and is not a $T_2$-essential.
	\end{theorem}
	\begin{proof} To prove that $f$ is not a $T_2$ essential we show that $
	\mathsf{P}\{\;\|T_2f\|_{L_2(\mathbb{R})}>0\;\}<1.
		$
		Since	$ [u_{2k};u_{2k+1}]\cap[u_{2j};u_{2j+1}]=\emptyset$   for any $k\ne j$ then, by \eqref{change},
		$$
		\mathsf{P}\{\;\|T_2f\|^2_{L_2(\mathbb{R})}>0\;\}=
		\mathsf{P}\left\{\int\limits_{-\infty}^{+\infty}\left(\sum\limits_{n=0}^{\infty}1\!\!\,{\rm I}_{[u_{2n};u_{2n+1}]}(x(u,2))\right)^2du>0\;\right\}=
		$$
		$$=
		\mathsf{P}\left\{\;\sum\limits_{n=0}^{\infty}\int\limits_{-\infty}^{+\infty}1\!\!\,{\rm I}_{[u_{2n};u_{2n+1}]}(x(u,2))du>0\;\right\}=
		$$
		$$=
		\mathsf{P}\left\{\;\sum_{n=0}^{\infty}(y(u_{2n+1},2)-y(u_{2n},2))>0\;\right\}=
		$$
		$$
		=\mathsf{P}\left\{\;\exists n\ge0:\; y(u_{2n+1},2)\ne y(u_{2n},2)\;\right\}\le	\sum\limits_{n=0}^{\infty}P\left\{\;y(u_{2n+1},2)\ne y(u_{2n},2)\;\right\}.
		$$
	Thus, by \eqref{tau},
		$$
		\sum\limits_{n=0}^{\infty}\mathsf{P}\left\{\;y(u_{2n+1},2)\ne y(u_{2n},2)\;\right\}=
	\sum\limits_{n=0}^{\infty}\frac{1}{\sqrt{4\pi}}\int\limits_{-\frac{1}{2^{n+1}}}^{\frac{1}{2^{n+1}}}e^{-\frac{v^2}{4}}dv
		\le\frac{1}{\sqrt{\pi}}<1.
		$$
		Consequently, the function	$f=\sum_{n=0}^{\infty}1\!\!\,{\rm I}_{[u_{2n};u_{2n+1}]}$ is not a $T_2$-essential.
		To prove that $f=\sum_{n=0}^{\infty}1\!\!\,{\rm I}_{[u_{2n};u_{2n+1}]}$ is a $T_1$-essential one can show the following estimation.
		\begin{lemma} Let $ \{w(u_n,\cdot)\}_{n=0}^{\infty}$ be a family of independent Wiener processes on $[0;1]$ such that $ w(u_n,0)=u_n$. Then for any $n\in\mathbb{N}$
			$$
			\mathsf{P}\left\{\;\max_{s\in[0;1]}\max_{j=\overline{0,2n-1}}w(u_j,s)\ge\min_{s\in[0;1]}w(u_{2n},s)\;\right\}<\frac{1}{2^{n^2}\sqrt{\pi \ln2}}.
			$$
		\end{lemma}
		\begin{proof} Let $w_1,w_2$ be an independent Wiener processes on $[0;1]$ started from point $0$, i.e. $w_1(0)=w_2(0)=0$. It can be noticed that
			$$
			\mathsf{P}\left\{\;\max_{s\in[0;1]}\max_{j=\overline{0,2n-1}}w(u_j,s)\ge\min_{s\in[0;1]}w(u_{2n},s)\;\right\}=
			$$
			$$
			=\mathsf{P}\left\{\;\exists\; j=\overline{0,2n-1}:\;\;\max_{s\in[0;1]}w(u_j,s)-\min_{s\in[0;1]}w(u_{2n},s)\ge0\;\right\}\le
			$$
			$$
			\le\sum_{j=0}^{2n-1}\mathsf{P}\left\{\;\max_{s\in[0;1]}w(u_j,s)-\min_{s\in[0;1]}w(u_{2n},s)\ge0\;\right\}\le
			$$
			$$
			\le \sum_{j=0}^{2n-1}\mathsf{P}\left\{\;\max_{s\in[0;1]}w_1(s)-\min_{s\in[0;1]}w_2(s)\ge u_{2n}-u_{j}\;\right\}.
			$$
			From the fact that $\{u_n\}_{n=0}^{\infty}$ is an increasing sequence we can estimate the last expression and complete the proof
			$$ \sum_{j=0}^{2n-1}\mathsf{P}\left\{\;\max_{s\in[0;1]}w_1(s)-\min_{s\in[0;1]}w_2(s)\ge u_{2n}-u_{j}\;\right\}
			\le
			$$
			$$\le\frac{1}{\sqrt{\pi}} \sum_{j=0}^{2n-1}\frac{1}{ u_{2n}-u_{j}}e^{-\frac{( u_{2n}-u_{j})^2}{4}}\le
			$$
			$$\le
			\frac{2n-1}{ \sqrt{\pi}(u_{2n}-u_{2n-1})}e^{-\frac{( u_{2n}-u_{2n-1})^2}{4}}\le
			\frac{1}{2^{n^2}\sqrt{\pi \ln2}}.
			$$ 
		\end{proof}
		
		
		\noindent Lets prove that the function $ f=\sum_{n=0}^{\infty}1\!\!\,{\rm I}_{[u_{2n};u_{2n+1}]}$ is a $T_1$-essential.
		Using reasoning from the first part of the proof it can be checked that for considered function $f$ the following equality holds			
		$$
		\mathsf{P}\{\;\|T_1f\|_{L_2(\mathbb{R})}>0\;\}=\mathsf{P}\left\{\sum_{n=0}^{\infty}(y(u_{2n+1},1)-y(u_{2n},1))>0\right\}.
		$$	
		
		Lets prove that 
		\begin{equation}\label{eq1}
		\mathsf{P}\left\{\;\limsup_{n\to\infty}\left(y(u_{2n+1},1)-y(u_{2n},1)\right)\ge1\;\right\}=1.
		\end{equation}
		
		\noindent Lets build a new processes	
		$\{\widetilde{y}(u_n,\cdot)\}_{n=0}^{\infty} $ such that  $\{\widetilde{y}(u_n,\cdot)\}_{n=0}^{\infty} $ and $\{y(u_n,\cdot)\}_{n=0}^{\infty} $ have the same distributions in $ \mathcal{C}\left([0;1]\right)^{\infty}$ in the following way \cite{4}. Let $ \{w(u_n,\cdot)\}_{n=0}^{\infty}$ be a given family of Wiener processes on $[0;1]$, $ w(u_n,0)=u_n$. Lets denote collision time of  $ f,g\in \mathcal{C}\left([0;1]\right)$ by $\tau[f,g]:=\inf\{t\;|\; f(t)=g(t)\}$. Put $\widetilde{y}(u_0,\cdot):= w(u_0,\cdot)$. Then for any $ n\in\mathbb{N}$, $s\in[0;1]$ one can define 					
		$$
		\widetilde{y}(u_n,s):= w(u_n,s)1\!\!\,{\rm I}\{\;s<\tau[w(u_n,\cdot),\widetilde{y}(u_{n-1},\cdot)]\;\}+
		$$
		$$
		+\widetilde{y}(u_{n-1},s)1\!\!\,{\rm I}\{s\ge\tau[w(u_n,\cdot),\widetilde{y}(u_{n-1},\cdot)]\;\}.
		$$
		According to constructions of stochastic processes $\{\widetilde{y}(u_n,\cdot)\}_{n=0}^{\infty} $	
		\begin{eqnarray}	\label{eq2}				
			\mathsf{P}\left\{\;\exists\;N\in\mathbb N:\;\forall\;n\ge N\quad \widetilde{y}(u_{2n},t)=w(u_{2n},t),\right.\nonumber\\
			\left.\;\widetilde{y}(u_{2n+1},t)=w(u_{2n+1},t)1\!\!\,{\rm I}\{\;t<\tau[w(u_{2n},\cdot),w(u_{2n+1},\cdot)]\}+\right.\\
			\left.+					w(u_{2n},t)1\!\!\,{\rm I}\{\;t\ge \tau[w(u_{2n},\cdot),w(u_{2n+1},\cdot)]\}\;\right\}=1.\nonumber
		\end{eqnarray}
		Thus, 
		$$\mathsf{P}\left\{\;\exists\;N\in\mathbb N:\;\forall\;n\ge N\quad \widetilde{y}(u_{2n+1},t)-\widetilde{y}(u_{2n},t)=w(u_{2n+1},t)-w(u_{2n},t)\right\}=1.$$
		For the considered sequence $\{u_n\}_{n=0}^{\infty}$ and any $n\in\mathbb N$ the following inequality holds	
		$$
	\mathsf{P}\{\;w(u_{2n+1},t)-w(u_{2n},t)\ge1\;\}=\int\limits_1^{\infty}\frac{1}{\sqrt{4\pi}}e^{-\frac{\left(v-\frac{1}{2^k}\right)^2}{4}}dv\ge 
		\frac{1}{\sqrt{4\pi}}\int\limits_1^{\infty}e^{-\frac{v^2}{4}}dv.
		$$
		\noindent Therefore, by the Borel-Cantelli lemma and \eqref{eq2},
		$$
		\mathsf{P}\{\;\limsup_{n\to\infty}(\widetilde{y}(u_{2n+1},t)-\widetilde{y}(u_{2n},t))\ge1\;\}=1.
		$$
	\end{proof}

	Using observation from Example \ref{ex1} one can introduce a family of $T_t$-essential functions for all $t>0$.
	
	For any $\varepsilon>0$ let us consider an integral operator $K_{\varepsilon}$ in $L_2(\mathbb{R})$ with the kernel
	\begin{equation}
		\label{eq30}
		k_{\varepsilon}(v_1,v_2)=\int_{\mathbb{R}}p_{\varepsilon}(u-v_1)p_{\varepsilon}(u-v_2)dy(u,t),
	\end{equation}
	where $ v_1,v_2\in\mathbb{R},$ and $p_{\varepsilon}(u)=\frac{1}{\sqrt{2\pi\varepsilon}}e^{-\frac{u^2}{2\varepsilon}}$. By the change of variables formula for an Arratia flow \cite{2},
	\begin{equation}
		\label{eq3}
		(K_{\varepsilon}f,f)=\int_{\mathbb{R}}(f\ast p_{\varepsilon})^2(x(u,t))du.
	\end{equation}
	
	\begin{lemma}
		For any $\varepsilon>0$ and nonzero function $f\in L_2(\mathbb{R})$
		$$
		\mathsf{P}\left\{\;\left(K_{\varepsilon}f,f\right)\ne0\;\right\}=1.
		$$
	\end{lemma}
	\begin{proof}
		According to \eqref{eq1} it is sufficient to note that $f\ast p_{\varepsilon} $ is an analytic function.
		Consequently, 
		for any $t>0$ the following relations are true
		$$
		\mathsf{P}\left\{\;\left(K_1f,f\right)>0\;\right\}=\mathsf{P}\left\{\; \|T_t(f\ast p_1)\|_{L_2(\mathbb{R})}>0\;\right\}=\mathsf{P}\left\{\;x(\mathbb{R},t)\cap Z_{f\ast p_1}=\emptyset \;\right\}=1.
		$$
	\end{proof}
	According to the last theorem and \eqref{eq3}, for any $\varepsilon>0$ and nonzero $f\in L_2(\mathbb{R})$ the function $f\ast p_{\varepsilon}$ is a $T_t$-essential for each $t>0$.
	
	\section{On change of compact sets under strong random operator generated by an Arratia flow}
	As it was noticed in the introduction any function with bounded support isn't a $T_t$-essential. Consequently,
	if $K\subseteq L_2(\mathbb{R})$ is a compact set of functions with uniformly bounded supports such that $T_t(K)$ is well-defined, then the image $T_t(K)$ equals to $ \{0\}$ with positive probability.
	It was shown in \cite{2} that $T_t$ may also change the geometry of $K$ even in the case of a compact set $K$ for which $ T_t(K)\ne \{0\}$ a.s. For example, the image $T_t(K)$ of a convergent sequence and its limiting point may not have limiting points. In this section we build a compact set $K$ for which $ T_t(K)\ne \{0\}$ a.s. and investigate the change of its Kolmogorov-widths in $L_2(\mathbb{R})$ under random operator $T_t$.
	\begin{definition}
		[\cite{5}]
		\label{defn2}
		The Kolmogorov $n$-width of a set $C\subseteq H$ in a Hilbert space $H$ is given by
		$$
		d_n(C)=\inf_{\dim L\leq n}\,\sup_{f\in C}\,\inf_{g\in L}\|f-g\|_H,
		$$
		where $L$ is a subspace of $H$.
	\end{definition}
	We consider the following compact set in $L_2(\mathbb{R})$
	\begin{equation}\label{eq7}
		K=\{\;f\in W_2^1(\mathbb{R})|\;\int_{\mathbb{R}}f^2(u)(1+|u|)^3du+\int_{\mathbb{R}}\left(f'(u)\right)^2(1+|u|)^7du\le1\;\}.
	\end{equation}	
	Estimations on its Kolmogorov-widths in $L_2(\mathbb{R})$ are presented in the next lemma.
	\begin{lemma}
		There exist positive constants $ C_1,C_2$ such that for any $n\in\mathbb N$
		$$
		\frac{C_1}{n}\le d_n\left(K\right)\le\frac{C_2}{n^{\frac{3}{10}}}.
		$$
	\end{lemma}
	\begin{proof}
		Let $n\in\mathbb N$ be fixed. To estimate $d_n\left(K\right)$ from above one can consider the partition $ \{u_k\}_{k=0}^{n}$ of $[-n^{\frac{1}{5}};n^{\frac{1}{5}}]$ into $n$ segments $ \{[u_k;u_{k+1}], k=\overline{0,n-1}\}$ with equal lengths.	Lets show that for the $n$-dimensional subspace\\ $ L_n=\overline{LS\{1\!\!\,{\rm I}_{[u_k;u_{k+1}]},\;k=\overline{0,n-1}\}}$
		$$
		\sup_{f\in K}\,\inf_{g\in L_n}\|f-g\|_{L_2(\mathbb{R})}\le \frac{C_2}{n^{\frac{3}{10}}}.
		$$ 
		If $f\in K$ then $\int_{\mathbb{R}}f^2(u)(1+|u|)^3du\le1$. Thus, for any $ C>0$
		$$
		\int_{|u|>c}f^2(u)du\le\frac{1}{(1+C)^3}\int_{|u|>c}f^2(u)(1+|u|)^3du\le\frac{1}{C^3}.
		$$
		So, for the function $ g_f=\sum\limits_{k=0}^{n-1}f(u_k)1\!\!\,{\rm I}_{[u_k;u_{k+1}]}\;\in L_n$ the following estimation is true
		$$
		\|f-g_f\|_{L_2(\mathbb{R})}^2\le\frac{1}{n^{\frac{3}{5}}}+\int_{|u|\le n^{\frac{1}{5}}}(f(u)-g_f(u))^2du.
		$$
		By the Cauchy inequality, for $f\in K $  and  $ u\in [u_k;u_{k+1}]$
		$$
		\left(\;\int\limits_{u_k}^uf'(v)dv\;\right)^2\le \int\limits_{u_k}^u\frac{dv}{(1+|v|)^7}\le u-u_k.
		$$
		Consequently, 
		$$
		\int\limits_{|u|\le n^{\frac{1}{5}}}\left(f(u)-g_f(u)\right)^2du=\sum\limits_{k=0}^{n-1}\int\limits_{u_k}^{u_{k+1}}\left(\;\int\limits_{u_k}^uf'(v)dv\;\right)^2du\le
		$$
		$$\le
		\frac{1}{2}\sum\limits_{k=0}^{n-1}(u_{k+1}-u_k)^2=\frac{2}{n^{\frac{3}{5}}},
		$$
		and the upper estimation for $ d_n(K)$ holds with the constant $C_2=3^{\frac{1}{2}}$. 
		
		To get a lower estimation for $d_n(K)$ we use the theorem about $n$-width of $(n+1)$-dimensional ball \cite{5}.  Let 
		$ \{u_{k}\}_{k=0}^{2(n+1)}$ be a partition of $[0;1]$ into $2(n+1)$ segments $ \{[u_k;u_{k+1}],\; k=\overline{0,2n+1}\}$ with equal lengths. Consider $(n+1)$-dimensional space $ L_{n+1}=LS\{f_k,\;k=\overline{0,n}\}$, where the functions $f_k$, $k=\overline{0,n}$, are defined as follows		
		\begin{equation}\label{eq4}	
			f_k=
			\begin{cases}
				\;	\;\;0,&\text{} u\notin [u_{2k};u_{2k+1}],\\
				\;	\;\;1,&\text{} u\in [u_{2k}+\frac{1}{6(n+1)};u_{2k}+\frac{2}{6(n+1)}],\\
				\;\;\;	6(n+1)(u-u_{2k}),&\text{} u\in[u_{2k};u_{2k}+\frac{1}{6(n+1)}],\\
				-6(n+1)(u-u_{2k+1}),&\text{} u\in[u_{2k}+\frac{2}{6(n+1)};u_{2k+1}].\\
			\end{cases}
		\end{equation}

		We show that if $c=\frac{2^3(5+2^9\cdot3^3)}{5}$ then the ball $ B_{n+1}=\{f\in L_{n+1}|\;\|f\|_{L_2(\mathbb{R})}\le\frac{1}{\sqrt{c}n}\}$ is a subset of $K$. 
		Since $\|f_k\|_{L_2(\mathbb{R})}^2=\frac{5}{18(n+1)}$, $k=\overline{0,n}$, then for any $f\in B_{n+1}$ such that $ f=\sum\limits_{k=0}^nc_kf_k$	
		the following  relation holds	$\sum_{k=0}^nc_k^2\le \frac{36}{5cn}$.	Thus, according to \eqref{eq4},
		$$
		\int\limits_{\mathbb{R}}f^2(u)(1+|u|)^3du+\int\limits_{\mathbb{R}}\left(f'(u)\right)^2(1+|u|)^7du\le
		$$
		$$
		\le2^3\|f\|_{L_2(\mathbb{R})}^2+2^7\cdot\sum\limits_{k=0}^nc_k^2\left(\int\limits_{u_{2k}}^{u_{2k}+\frac{1}{6(n+1)}}(6(n+1))^2du+
		\int\limits_{u_{2k}+\frac{2}{6(n+1)}}^{u_{2k+1}}(6(n+1))^2du\right)\le
		$$
		$$
		\le			\frac{2^3}{cn^2}+2^{10}\cdot 3n\frac{36}{5cn}\le\frac{1}{c}\cdot\frac{2^3(5+2^9\cdot3^3)}{5}=1.
		$$	
		Consequently, $B_{n+1}\subset K$ and $ d_n(K)\ge d_n(B_{n+1})$. Due to the theorem about $n$-width of $(n+1)$-dimensional ball, $d_n(B_{n+1})=\frac{1}{\sqrt{c}n}$ \cite{5}. So the lower estimation for		$ d_n(K)$ holds with $ C_1:=\sqrt{c}$.
	\end{proof}
	
	To show that estimations from above for the Kolmogorov-widths of considered compact set $K$ don't change under $T_t$ one may use the same idea as in Lemma 2. 
	\begin{theorem}
		There exists $\widetilde{\Omega}$ of probability one such that for any $\omega\in\widetilde{\Omega}$ and $n\in\mathbb{N}$
		\begin{equation}\label{eq6}	
			d_n\left(T_t^{\omega}(K)\right)\le\frac{C(\omega)}{n^{\frac{3}{10}}},
		\end{equation}
		where the constant $C(\omega)>0$ doesn't depend on $n$.
	\end{theorem}
	\begin{proof} For a fixed $n\in\mathbb N$ lets consider a partition $\{u_k\}_{k=0}^n$ of $[-n^{\frac{1}{5}};n^{\frac{1}{5}}]$ into $n$ segments with equal lengths.
		To prove \eqref{eq6} it's sufficient to show the following inequality for the linear space $L_n^{\omega}=LS\{\;T_t^{\omega}1\!\!\,{\rm I}_{[u_k;u_{k+1}]},\;k=\overline{0,n-1}\;\}$ with dimension at most $n$ 
		$$
		\sup_{h_1\in T_t^{\omega}(K)}\inf_{h_2\in L_n^{\omega}}\|h_1-h_2\|_{L_2(\mathbb{R})}\le\frac{C(\omega)}{n^{\frac{3}{10}}}.
		$$ 
		According to the change of variable formula for an Arratia flow, one can check the equality for any $f\in K$ 
		$$
		\left\|T_t^{\omega}f-T_t^{\omega}\left(\sum_{k=0}^{n-1}f(u_k)1\!\!\,{\rm I}_{[u_k;u_{k+1}]}\right)\right\|^2_{L_2(\mathbb{R})}=
		\int_{|u|>n^{\frac{1}{5}}}f^2(u)dy(u,t,\omega)+
		$$
		$$+\int_{|u|\le n^{\frac{1}{5}}}\left(f(u)-\sum_{k=0}^{n-1}f(u_k)1\!\!\,{\rm I}_{[u_k;u_{k+1}]}(u)\right)^2dy(u,t,\omega).
		$$
		To estimate from above the last integral lets notice that  
		$$
		\int_{|u|\le n^{\frac{1}{5}}}\left(f(u)-\sum_{k=0}^{n-1}f(u_k)1\!\!\,{\rm I}_{[u_k;u_{k+1}]}(u)\right)^2dy(u,t,\omega)\le
		$$
		$$\le
		\sum_{k=0}^{n-1}\int_{u_k}^{u_{k+1}}\left(\int_{u_k}^u\left|f'(v)\right|dv\right)^2dy(u,t,\omega).
		$$
		Due to \eqref{eq7}, for any $f\in K$ and $u\in [u_k;u_{k+1}]$ 
		$$
		\left(\int_{u_k}^u\left|f'(v)\right|dv\right)^2\le \int_{u_k}^{u}\frac{dv}{(1+|v|)^7}\le u_{k+1}-u_k.
		$$ 
		Thus, 
		$$
		\sum_{k=0}^{n-1}\int_{u_k}^{u_{k+1}}\left(\int_{u_k}^u\left|f'(v)\right|dv\right)^2dy(u,t,\omega)\le
		$$
		$$\le
		\sum_{k=0}^{n-1}(u_{k+1}-u_k)\int_{u_k}^{u_{k+1}}dy(u,t,\omega)=\frac{2}{n^{\frac{4}{5}}}(y(n^{\frac{1}{5}},t,\omega)-y(-n^{\frac{1}{5}},t,\omega)).
		$$
		For an Arratia flow  $ \{y(u,s),u\in\mathbb{R},s\in [0;t]\}$ the following relation is true \cite{7}
		$$
		\lim\limits_{|u|\to+\infty}\frac{|y(u,t)|}{|u|}=1\;\;\mbox{a.s.}
		$$
		Consequently, for any  $\omega\in\widetilde{\Omega}=\{\omega'\in\Omega|\; \lim\limits_{|u|\to+\infty}\frac{|y(u,t,\omega')|}{|u|}=1\}$ the estimation holds
		\begin{equation}\label{eq9}
			\int_{|u|\le n^{\frac{1}{5}}}\left(f(u)-\sum_{k=0}^{n-1}f(u_k)1\!\!\,{\rm I}_{[u_k;u_{k+1}]}(u)\right)^2dy(u,t,\omega)	\le
			\frac{4c(\omega)}{n^{\frac{3}{5}}}
		\end{equation}
		with the constant	
		\begin{equation}\label{eq8}
			c(\omega):=\sup\limits_{|u|\ge1}\frac{|y(u,t,\omega)|}{|u|}.
		\end{equation} 	 
		
		Lets prove that for any $\omega\in\widetilde{\Omega}$ there exists a constant $\widetilde{c}(\omega)$ such that
		$$\int_{|u|>n^{\frac{1}{5}}}f^2(u)dy(u,t,\omega)\le\frac{\widetilde{c}(\omega)}{n^{\frac{3}{5}}}.$$ 
		It can be noticed that 
		$
		\int_{|u|>n^{\frac{1}{5}}}f^2(u)dy(u,t)\le	\frac{1}{n^{\frac{3}{5}}}\int_{|u|>n^{\frac{1}{5}}}f^2(u)(1+|u|)^3dy(u,t).
		$
		Denote by $\{\theta_j\}_{j=1}^{\infty}$ a sequence of jump points of the function $ y(\cdot,t)$ on $\mathbb{R}_+$. Thus, one may show
		$$
		\int_{u>n^{\frac{1}{5}}}f^2(u)(1+u)^3dy(u,t)=
		\sum_{\theta_i\ge n^{\frac{1}{5}}}f^2(\theta_i)(1+\theta_i)^3\Delta y(\theta_i,t)=
		$$
		$$
		=	\sum_{k=1}^{\infty}\sum\limits_{\{i:\;\theta_i\in\left[\left. k;k+1\right)\right.\}}f^2(\theta_i)(1+\theta_i)^3\Delta y(\theta_i,t)
		\le
		$$
		$$\le\sum_{k=1}^{\infty}(2+k)^3\sum_{\{i:\;\theta_i\in\left[\left.k;k+1\right)\right.\}}f^2(\theta_i)\Delta y(\theta_i,t).
		$$
		According to the Cauchy inequality and \eqref{eq7}, for any	$u\in\mathbb{R}_+$ the following relations hold	
		$$
		f^2(u)\le\int_u^{+\infty}\left(f'(v)\right)^2(1+v)^7dv\cdot\int_u^{+\infty}\frac{dv}{(1+v)^7}\le\frac{1}{6u^6}.
		$$
		Consequently, due to \eqref{eq8}, the inequalities are true	
		$$
		\sum_{k=1}^{\infty}(2+k)^3\sum_{\{i:\;\theta_i\in\left[\left.k;k+1\right)\right.\}}f^2(\theta_i)\Delta y(\theta_i,t)\le
		$$
		$$\le
		\sum_{k=1}^{\infty}(2+k)^3\frac{1}{6k^6}(y(k+1,t)-y(k,t))\le
		\frac{16c}{3}\sum_{k=1}^{\infty}\frac{1}{k^2}.
		$$
		Hence, for any $\omega\in\widetilde{\Omega}$ there exists the constant $ C_1(\omega)=\frac{16c(\omega)}{3}$ such that
		$$\int\limits_{u>n^{\frac{1}{5}}}f^2(u)dy(u,t,\omega)\le \frac{C_1(\omega)}{n^{\frac{3}{5}}}.$$ Similarly, it can be proved that   $\int\limits_{u<-n^{\frac{1}{5}}}f^2(u)dy(u,t,\omega)\le \frac{C_1(\omega)}{n^{\frac{3}{5}}} $. According to this and \eqref{eq9}, for any $\omega\in\widetilde{\Omega}$ an upper estimation for $d_n\left(T_t^{\omega}(K)\right) $ is true.  
	\end{proof}	 
	
	The functions from Lemma 2 that were used to build the $(n+1)$-dimensional subspace are not $T_t$ essential for any $t>0$. Thus, the image of this subspace under the random operator $T_t$ may be equal to $\{0\}$ with positive probability. So, one can ask about existence of a finite-dimensional subspace such that for any $t>0$ its image under $T_t$ is a linear subspace with the same dimension. 
	
	\section{On subspace that preserves dimension under random operator generated by an Arratia flow}
	In this section for any $t>0$ and $ n\in\mathbb{N}$ we present a family $ \{g_k,k=\overline{0,n}\}$ of linearly independent $T_t$-essential functions such that its images under  $T_t$ are linearly independent. Such family generates a subspace which preserves dimension under random operator generated by an Arratia flow. It can be used to get a lower estimation of $ d_n(T_t(K)).$ 
	
	Lets fix any $n\in\mathbb{N}$, and build a family of $(n+1)$ linearly independent functions in the following way. Let $\{u_k\}_{k=0}^{2(n+1)}$ be a partition of	$[0;n^{-2}]$ into $2(n+1)$ segments with equal lengths. For any $k=\overline{0,n}$ define $f_k$ by	
	\begin{equation}\label{eq10}	
		f_k=
		\begin{cases}
			\;	\;\;0,&\text{} u\notin [u_{2k};u_{2k+1}],\\
			\;	\;\;1,&\text{} u\in [u_{2k}+\frac{n^{-2}}{6(n+1)};u_{2k}+\frac{2n^{-2}}{6(n+1)}],\\
			\;\;\;	\frac{6(n+1)}{n^{-2}}(u-u_{2k}),&\text{} u\in[u_{2k};u_{2k}+\frac{n^{-2}}{6(n+1)}],\\
			-\frac{6(n+1)}{n^{-2}}(u-u_{2k+1}),&\text{} u\in[u_{2k}+\frac{2n^{-2}}{6(n+1)};u_{2k+1}].\\
		\end{cases}
	\end{equation}	
	
	\begin{lemma}
		There exists  $\varepsilon_0>0$ such that for any $0<\varepsilon\le\varepsilon_0$ the functions	$ \{f_k\ast p_{\varepsilon}, k=\overline{0,n}\}$ are linearly independent.
	\end{lemma}
	\begin{proof}
		Since considered functions $ \{f_k, k=\overline{0,n}\}$ are linearly independent then its Gram determinant doesn't equal to 0, i.e.  $ G(f_0,\ldots, f_n)\ne0$. For each $ k=\overline{0,n} $ 
		$$ f_k\ast p_{\varepsilon}\to f_k,\; \varepsilon\to0.$$ Hence, due to continuity of Gram determinant, one may notice that there exists $\varepsilon_0>0$ such that for any $0<\varepsilon\le\varepsilon_0$
		$$
		G(f_0\ast p_{\varepsilon},\ldots,f_n\ast p_{\varepsilon})\ne0,
		$$
		and the desired result is proved.
	\end{proof}
	
	\begin{theorem}
		There exists a set	$\Omega_0$ of probability one such that for any	$\omega\in\Omega_0$ the functions 		
		$ T_t^{\omega}(f_0\ast p_{\varepsilon}),\ldots,T_t^{\omega}(f_n\ast p_{\varepsilon})$ are linearly independent.
	\end{theorem}
	\begin{proof}
	
		Denote by $K_{\varepsilon}$ the integral operator in 	$L_2(\mathbb{R})$ with the kernel $k_{\varepsilon}$. 
		To prove the statement of the theorem it's enough to show that on some $\Omega_0$ of probability one the following inequality holds	$ (K_{\varepsilon}f,f)>0$ for any nonzero $f\in LS\{f_0,\ldots,f_n\}$. Due to \eqref{change} 			
		\begin{equation}
			\label{eq12}
			(K_{\varepsilon}f,f)=\sum_{\theta}(f\ast p_{\varepsilon})^2(\theta)\Delta y(\theta,t),
		\end{equation}
		where $ \theta$ is a point of jump of the function 	$ y(\cdot,t)$.
		
	It was proved in \cite{11} that there exists $\Omega_0$ of probability one such that for any $\omega\in\Omega_0$ a linear span of the functions $\{p_{\varepsilon}(\cdot-\theta(\omega))|_{[0;1]}\}_{\theta(\omega)}$ is dense in $L_2([0;1])$. Thus, on the set $\Omega_0$ for any $f\in LS\{f_0,\ldots,f_n\}\subset L_2([0;1])$ one can find a random point $\theta_f$ such that $ (f(\cdot),p_{\varepsilon}(\cdot-\theta_f))\ne0.$ Since $y(\cdot,t):\mathbb{R}\to\mathbb{R}$ is nondecreasing then $\Delta y(\theta,t)>0 $ for any jump-point $\theta$. Consequently, on the set $\Omega_0$
	$$
	\sum_{\theta}(f\ast p_{\varepsilon})^2(\theta)\Delta y(\theta,t)=\sum_{\theta}(f(\cdot),p_{\varepsilon}(\cdot-\theta))^2 \Delta y(\theta,t)\ge
	$$
	$$
	\ge (f(\cdot),p_{\varepsilon}(\cdot-\theta_f))^2 \Delta y(\theta_f,t)>0,
	$$
	which proves the theorem.
	\end{proof}

\par\bigskip\noindent
{\bf Acknowledgment.} The first author acknowledges financial support from the Deutsche Forschungsgemeinschaft (DFG) within the project "Stochastic Calculus and Geometry of Stochastic Flows with Singular Interaction" for initiation of international collaboration between the Institute of Mathematics of the Friedrich-Schiller University Jena (Germany) and the Institute of Mathematics of the National Academy of Sciences of Ukraine, Kiev. 

The second author is grateful to the Institute of Mathematics of the Friedrich-Schiller University Jena (Germany) for hospitality and financial support within Erasmus+ programme 2015/16-2017/18 between Jena Friedrich-Schiller University (Germany) and the Institute of Mathematics of the National Academy of Sciences of Ukraine, Kiev.

\bibliographystyle{amsplain}


\end{document}